\documentclass[12pt]{amsart}
\usepackage{ amsmath, amsthm, amsfonts, amssymb, color}
 \usepackage{mathrsfs}
\usepackage{amsfonts, amsmath}
 \usepackage{amsmath,amstext,amsthm,amssymb,amsxtra}
 \usepackage{txfonts} %also pxfonts
 \usepackage[colorlinks, citecolor=blue,pagebackref,hypertexnames=false]{hyperref}
 \allowdisplaybreaks
 \usepackage{pgf,tikz}
 \usepackage{multirow}%newtablepackage1
 \usepackage{diagbox} %newtablepackage2

 \usepackage{tikz}

\usetikzlibrary{decorations.pathreplacing}

\begin{document}

 \baselineskip 16.6pt
\hfuzz=6pt

\widowpenalty=10000

\newtheorem{cl}{Claim}
\newtheorem{theorem}{Theorem}[section]

\newtheorem{proposition}[theorem]{Proposition}
\newtheorem{corollary}[theorem]{Corollary}
\newtheorem{lemma}[theorem]{Lemma}

\theoremstyle{definition}
\newtheorem{definition}[theorem]{Definition}

\newtheorem{assum}{Assumption}[section]
\newtheorem{example}[theorem]{Example}
\newtheorem{Pro}[theorem]{Problem}
\newtheorem{remark}[theorem]{Remark}
\renewcommand{\theequation}
{\thesection.\arabic{equation}}

\def\SL{\sqrt H}

\newcommand{\mar}[1]{{\marginpar{\sffamily{\scriptsize
        #1}}}}

\newcommand{\as}[1]{{\mar{AS:#1}}}

\newcommand\R{\mathbb{R}}
\newcommand\RR{\mathbb{R}}
\newcommand\CC{\mathbb{C}}
\newcommand\NN{\mathbb{N}}
\newcommand\ZZ{\mathbb{Z}}
\newcommand\HH{\mathbb{H}}
\def\RN {\mathbb{R}^n}
\renewcommand\Re{\operatorname{Re}}
\renewcommand\Im{\operatorname{Im}}

\newcommand{\mc}{\mathcal}
\newcommand\D{\mathcal{D}}
\def\hs{\hspace{0.33cm}}
\newcommand{\la}{\alpha}
\def \l {\alpha}
\newcommand{\eps}{\tau}
\newcommand{\pl}{\partial}
\newcommand{\supp}{{\rm supp}{\hspace{.05cm}}}
\newcommand{\x}{\times}
\newcommand{\lag}{\langle}
\newcommand{\rag}{\rangle}

\newcommand\wrt{\,{\rm d}}

% 字母编号的定理环境（独立计数器，纯字母编号）
\newtheorem{theoremA}{Theorem}
\renewcommand{\thetheoremA}{\Alph{theoremA}} % 如Theorem A, Theorem B

\title[]{On the relationship between block spaces and Orlicz spaces}
\maketitle
%\author{Yanping Chen$^*$}\thanks{$*$ Corresponding Author}
%\author{Teng Wang}
%\author{Huoxiong Wu}
%
%\address{Yanping Chen: Department of Mathematics, Northeastern University, Shenyang 110004, China}
%\email{yanpingch@126.com}
%
%\address{Teng Wang: School of Mathematics and Physics, University of Science and Technology Beijing, Beijing 100083, China}
%\email{tengwustb@126.com}
%
%\address{Huoxiong Wu: School of Mathematical Sciences, Xiamen University, Xiamen, Fujian
%361005, People's Republic of China}
%\email{huoxwu@xmu.edu.cn}

%\vspace{-0.5cm}
%\begin{center}
%
%{\bf Yanping Chen \footnote{\small {Corresponding author.\ }}}\\
%Department of Mathematics,\\
%Northeastern University,\\
%Shenyang 110004, China\\
%E-mail: {\it yanpingch@126.com}
%\vskip 0.1cm
%{\bf Zhengyang Ji }\\
%Department of Mathematics,\\
%Northeastern University,\\
%Shenyang 110004, China\\
%E-mail: {\it zhengyang.ji.zjut@gmail.com}
%\vskip 0.1cm
%{\bf Teng Wang}\\
%School of Mathematics and Physics,\\
% University of Science and Technology Beijing,\\
%Beijing 100083,  China \\
%E-mail: {\it tengwustb@126.com} 
%\vskip 0.1cm
%{\bf Huoxiong Wu}\\
%School of Mathematical Sciences,\\
%Xiamen University,\\
%Xiamen, Fujian 361005, China \\
%E-mail: {\it huoxwu@xmu.edu.cn} 
%\vskip 0.1cm
%\end{center}
%\vspace{-1cm}

\begin{center}
\author{%
    Yanping Chen$^{1,*}$\qquad 
    Zhengyang Ji$^{1}$\qquad 
    Teng Wang$^{2}$\qquad 
    Huoxiong Wu$^{3}$
}
\end{center}

% 在 maketitle 之后、abstract 之前，用 \footnotetext 手动定义脚注内容
\footnotetext[1]{%
    Department of Mathematics, Northeastern University, Shenyang 110004, China. \\
    E-mail: Yanping Chen: \texttt{yanpingch@126.com}; 
    Zhengyang Ji: \texttt{zhengyang.ji.zjut@gmail.com}.}
\footnotetext[2]{%
    School of Mathematics and Physics, University of Science and Technology Beijing, 
    Beijing 100083, China. \\
    E-mail: \texttt{tengwustb@126.com}.}
\footnotetext[3]{%
    School of Mathematical Sciences, Xiamen University, Xiamen, Fujian 361005, China. \\
    E-mail: \texttt{huoxwu@xmu.edu.cn}.}
% 关键：用 \renewcommand 临时改一下脚注格式，让 [1] 显示为 *
{
\renewcommand{\thefootnote}{\fnsymbol{footnote}}
\footnotetext[1]{Corresponding author.}
}
% 大括号的作用：格式改动的范围只在大括号内部，不影响外面的脚注
\date{\today}

\renewcommand{\thefootnote}{}
\footnote{The project was in part supported by: Yanping Chen's
		National Natural Science Foundation of China (\# 12525105, \# 12371092) and Huoxiong Wu's
		National Natural Science Foundation of China (\# 12171399, \# 12271041).}


\vspace{-1cm}
\begin{abstract}
Let $1<q\le \infty$ and $v>-1$. Let $(X,d,\mu)$ be an $s$-Ahlfors-regular quasi-metric measure space.
Suppose that  $B^{0,v}_q(X)$ is the block space  which consists of all functions that admit a decomposition into $q$-blocks supported on balls.
In this paper, we study the relationship between the block space $B^{0,v}_q(X)$ and the Orlicz-type space
$L(\log^+\!\!L)^{1+v}(X)$.

More precisely, we show that the block space $B_q^{0,v}(X)$ is a proper subspace of the Orlicz space $L(\log^+\!\!L)^{1+v}(X)$ for any fixed  $1<q\le \infty$ and $v>-1$. Namely,
$$B_q^{0,v}(X)\subsetneq L(\log^+\!\!L)^{1+v}(X),$$ 
which gives a confirmed answer to  a longstanding open problem concerning the relationship between block spaces and Orlicz-type spaces on the unit sphere $\mathbb S^{n-1}$.
We further show that $L(\log^+\!\!L)^{1+v}(X)$   is  the smallest  Orlicz-type space containing  $B^{0,v}_{q}(X)$.

We also introduce a generalized block space $\mathscr B_q^{0,v}(X)$ that depends only on the measure structure and show that this space is equivalent to the Orlicz space $L(\log^+\!\!L)^{1+v}(X)$ when $\mu(X)<\infty$.

Finally,  we consider two special cases that further clarify the roles of the parameter $q$ and the logarithmic weight.

\end{abstract}
	\maketitle

\arraycolsep=1pt

\section{Introduction}

%开篇点明主题

%定义部分

Let $(X,d,\mu)$ be an $s$-Ahlfors-regular quasi-metric measure space (see \cite{Alv15, Dav97}).
That is, $X$ is a set, $d:X\times X\to[0,\infty)$ is a quasi-metric satisfying
\begin{enumerate}
\item $d(x,y)=d(y,x)$ for all $x,y\in X$;
\item $d(x,y)=0$ if and only if $x=y$;
\item there exists a constant $A\ge1$ such that
\[
d(x,y)\le A\bigl(d(x,z)+d(z,y)\bigr) \text{~for all~} x,y,z\in X. 
\]
\end{enumerate}
Moreover, $\mu$ is a nonzero Borel regular measure on $X$ for which there exist constants
$c,C,s>0$ such that
\begin{equation}\label{upc}
c\,r^s
\le
\mu(B(x,r))
\le
C\,r^s,
\qquad
x\in X,\quad
0<r<\operatorname{diam}(X),
\end{equation}
where
$
B(x,r)=\{y\in X:d(x,y)<r\}
$
and
$
\operatorname{diam}(X)
=
\sup \{d(x,y):x,y\in X\}.
$

The class of $s$-Ahlfors-regular quasi-metric measure spaces
provides a unified framework for many settings.
It includes, in particular, the Euclidean space $\mathbb R^n$,
the unit sphere $\mathbb S^{n-1}$,
and homogeneous groups equipped with their Haar measures.

Let $1<q\le \infty$.
A measurable function $a$ on $X$ is called a \emph{$q$-block} if there exists a ball $B \subset X$ such that
$0<\mu(B)<\infty $ and
\[
\supp a \subseteq B, \quad  \|a\|_{L^q(B)} \le \mu(B)^{-1/q'}.
\]
Thus a $q$-block is a function supported on a ball whose $L^q$-norm is normalized according to the measure of that ball.

For $k\ge 0$ and $v>-1$, define the function
\[
\phi_{k,v}(t) = \chi_{(0,1)}(t) \int_t^1 s^{-1-k}\bigl(\log^+(1/s)\bigr)^v\,ds, \quad t>0,
\]
where $\log^+ t = \max\{\log t,0\}$. With $\phi_{k,v}$ at hand, we now introduce the block spaces $B^{k,v}_q(X)$.

\begin{definition}[Block space $B^{k,v}_q(X)$]  {\it Let $1<q\le \infty$, $k\ge 0$ and $v>-1.$
The block space $B^{k,v}_q(X)$ consists of all functions $f \in L^1(X)$ that admit a decomposition
\[
f = \sum_{j=1}^\infty \lambda_j a_j,
\]
where each $a_j$ is a $q$-block supported on a ball $B_j$, and 
\[
\sum_{j=1}^\infty |\lambda_j|\, \bigl(1+\phi_{k,v}(\mu(B_j))\bigr) < \infty.
\]}
\end{definition}

The space $B^{k,v}_q(X)$ is equipped with the norm
\[
\|f\|_{B^{k,v}_q(X)} := \inf \Big\{ \sum_{j=1}^\infty |\lambda_j| \Bigl( 1+\phi_{k,v}\big(\mu(B_j)\big) \Bigr) \Big\},
\]
where the infimum is taken over all decompositions $ f=\sum_{j=1}^\infty \lambda_j a_j$ such that $\lambda_j\in\mathbb C$ and each $a_j$ is a $q$-block supported on a ball $B_j$.
It is straightforward to verify that
$B^{k,v}_q(X)$
is a Banach space.

Roughly speaking, the space $B^{k,v}_q(X)$ consists of functions that admit a decomposition into $q$-blocks supported on balls, where the logarithmic weight $\phi_{k,v}$ controls the contribution of $q$-blocks to the norm via the measure of those balls.
 
%来源背景
The method of block decomposition of functions in $\mathbb R^n$  originated in the work of
Taibleson and  Weiss  on the convergence of the Fourier series in connection with the developments of the real Hardy spaces
(see \cite{Taib83}). Since then, many applications of the block decomposition in harmonic
analysis were discovered (see \cite{FZ17, Lu84, LuWang92, MTW85, Zalo88}). In particular, Lu, Taibleson and Weiss \cite{LuTaW82} established the convergence of Bochner--Riesz means on certain block spaces. For more details, readers may see the book \cite{LuTaW89}, which is a good reference on this topic. 
On the other hand, Blasco, Ruiz and Vega \cite{Bla99} introduced block spaces as preduals of Morrey spaces, and this theory was recently extended to spaces of homogeneous type in \cite{Dun24}.
Particularly, one can find 
in \cite{LuTaW89} (see also \cite{Lu84}) that 
$$ B^{k,v}_q(X)\subseteq  B^{k,u}_q(X),\,\,\,\hbox{for}\,\,v>u,k\in \mathbb R; $$ and
$$B^{k,\alpha}_q(X)\subseteq B^{j, \beta}_q(X)\subseteq  B^{0,\gamma}_q(X) $$ 
for $0<j<k$ and any real numbers $\alpha,\,\beta,\,\gamma$.
%本文真正研究的

When
$(X,d,\mu)=
(\mathbb S^{n-1},|\cdot|,d\sigma)$,
the space
$B_q^{k,v}(X)$
reduces to the block spaces on the unit sphere $\mathbb S^{n-1}$ introduced by Jiang and Lu \cite{JiangLu92} in the study of the homogeneous singular integral operators. For further information about the theory of spaces generated by blocks  on the unit sphere $\mathbb S^{n-1}$ for singular integrals and related operators, see
\cite{AlH99,AAP04, ACP11, AP02, CL11,CLW25,JiangLu93,LiY12,LuWu04,LuWu04b, Wu05,YuLu16}
or the survey \cite{Lu07}.

Now, we recall the definition of the Orlicz-type space $L(\log^+\!\! L)^\beta(X)$.

\begin{definition}[Orlicz-type space $L(\log^+\!\!  L)^\beta(X)$]
{\it Let $\beta>0$.
The space
$L(\log^+\!\!  L)^\beta(X)$
consists of all measurable functions $f$ on $X$ satisfying
$$
\|f\|_{L(\log^+\!\!  L)^\beta(X)}<\infty,
$$ where the quasi-norm $\|f\|_{L(\log^+\!\!  L)^\beta(X)}$ is defined by
\[
\|f\|_{L(\log^+\!\!  L)^\beta(X)} :=\int_X|f(x)|\bigl(\log(2+|f(x)|)\bigr)^\beta\,d\mu(x).
\]}
\end{definition}

The space $L(\log^+ \!\!L)^\beta(X)$ belongs to the class of Orlicz spaces introduced by Orlicz in 1932 as a natural extension of the $L^p$ spaces, see, e.g., \cite{Mu83,RaoRen91}. 
Among them, the Zygmund space $L\log^+\!\!L(X)$, corresponding to a logarithmic Young function, emerged naturally in Fourier analysis and differentiation theory, especially in the work of Zygmund on trigonometric series and singular integrals \cite{Zygmund02}. 
Later, the spaces $L(\log^+\!\! L)^\beta(X)$ played an important role in harmonic analysis, particularly in the study of the convergence and growth of partial sums of Fourier series \cite{Car66, Lie13}, as well as in the boundedness of Calder\'on--Zygmund singular integrals with rough kernels \cite{AACP02, CZ56, Chr88, CorF75, Gra14a, Gra14b, Seeger96, T, Wal72}.
More developments in this direction can also be found in the works of Stein and Weiss on singular integrals and real-variable methods (see, e.g., \cite{Stein70,Stein93}).

In 1992, Jiang and Lu \cite{JiangLu92} established the following inclusion relations  on the unit sphere $\mathbb S^{n-1}$.
\begin{theoremA}(\cite{JiangLu92})
{\it Let $1<q\leq \infty, k\geq 0$ and $v>-1$.
Then the following holds:}
\begin{enumerate}
\item $L^q(\mathbb{S}^{n-1})\subseteq B^{k,v}_q(\mathbb{S}^{n-1});$
\item $B^{k,v_2}_q(\mathbb{S}^{n-1})\subseteq B^{k,v_1}_q(\mathbb{S}^{n-1}),  \quad v_2>v_1>-1;$
\item $B^{k_2,v_2}_q(\mathbb{S}^{n-1})\subseteq B^{k_1,v_1}_q(\mathbb{S}^{n-1}), \quad 0\leq k_1<k_2, \, v_2,v_1>-1;$
\item $B^{k,v}_{q_2}(\mathbb{S}^{n-1})\subseteq B^{k,v}_{q_1}(\mathbb{S}^{n-1}),\quad 1<q_1<q_2.$
\end{enumerate}
\end{theoremA}

In 1993, Keitoku and Sato \cite{Kei93} further clarified the relationship between $B^{k,v}_q(\mathbb{S}^{n-1})$ and $L^p(\mathbb{S}^{n-1})$.
\begin{theoremA}(\cite{Kei93})\label{kei}
{\it Let $1<q\leq \infty, k\geq 0$ and $v>-1$. 
Then the following holds:}
\begin{enumerate}

\item $B^{k,v}_q(\mathbb{S}^{n-1})\subseteq L^p(\mathbb{S}^{n-1}),\quad  1<p\leq q,\, k>1/p';$
\item $B^{k,v}_q(\mathbb{S}^{n-1})= L^q(\mathbb{S}^{n-1}),\quad  k>1/q'\,\text{~or~}\, k=1/q'\, \text{~ \it and~}\, v\geq 0;$
\item $\cup_{p>1} L^p(\mathbb{S}^{n-1}) \subseteq B^{0,0}_q(\mathbb{S}^{n-1}),\quad q>1.$

\end{enumerate}
\end{theoremA}

Moreover,  the authors of \cite{Kei93} also pointed  out that
 for every $\epsilon>0$,
\begin{equation}\label{KS}
B_q^{0,v}(\mathbb S^{n-1})
\nsubseteq
L(\log^+\!\!L)^{1+v+\epsilon}(\mathbb S^{n-1}),
\qquad 1<q\leq \infty, v>-1.
\end{equation} Especially,
\begin{equation}\label{KS1}
B_q^{0,0}(\mathbb S^{n-1})
\nsubseteq
L(\log^+\!\!L)^{1+\epsilon}(\mathbb S^{n-1}),
\qquad 1<q\leq \infty.
\end{equation}
By (3) in Theorem \ref{kei} and \eqref{KS1}, it is easy to see that the block space $B^{0,0}_q(\mathbb{S}^{n-1})$ is an interesting space, which is quite close to $L^1(\mathbb{S}^{n-1})$.

A significant advance was later obtained by Ye and Zhu \cite{YZ06}, who proved that
\begin{equation}\label{YZ}
B_q^{0,v}(\mathbb S^{n-1}) \subseteq H^1(\mathbb S^{n-1}) + L(\log^+\!\!L)^{1+v}(\mathbb S^{n-1}),\qquad 1<q\leq \infty, v>-1,
\end{equation}
where $H^1(\mathbb S^{n-1})$ is the Hardy space on the unit sphere  $\mathbb S^{n-1}$ (see e.g., \cite{CW77, Col82}).
In particular, every function in $ L(\log^+\!\!L)(\mathbb S^{n-1})$
with vanishing integral belongs to $H^1(\mathbb S^{n-1})$.

The results of \eqref{KS} and \eqref{YZ} suggest that 
$B_q^{0,v}(\mathbb S^{n-1})$
 and 
$L(\log^+\!\!L)^{1+v}(\mathbb S^{n-1})$ seem to have some kind of relationship.
Nevertheless, even in the simplest case $v=0$,
the precise relationship between $B_q^{0,0}(\mathbb S^{n-1})$ and $L\log^+\!L(\mathbb S^{n-1})$  remains open for a long time.
In particular, Lu explicitly raised this problem in his survey article \cite{Lu07} (see also \cite{AlH99,LuWu04} for earlier observations).

\begin{Pro}\label{pro} {\it Let $1<q\leq \infty,\,\, v>-1.$
 Whether $$B_q^{0,v}(\mathbb S^{n-1})=L(\log^+\!\!L)^{1+v}(\mathbb S^{n-1}),$$ or $$B_q^{0,v}(\mathbb S^{n-1})\subsetneq L(\log^+\!\!L)^{1+v}(\mathbb S^{n-1}),$$  or $$ L(\log^+\!\!L)^{1+v}(\mathbb S^{n-1})\subsetneq  B_q^{0,v}(\mathbb S^{n-1})?$$}
\end{Pro}

%问题重要性
The solution to Problem \ref{pro} is crucial for understanding the  significance and the role of block spaces  on the unit sphere $\mathbb S^{n-1}$. Without a precise description of their relationships with Orlicz-type spaces, it remains unclear whether results established in block spaces  genuinely go beyond the $L(\log^+\!\!L)^{1+v}$ framework.

%This problem is very important in the development of block spaces. 
%Indeed, in the absence of a precise understanding of the relationship between block and Orlicz-type spaces, it remains unclear whether the boundedness results in block spaces genuinely extend beyond the  $L(\log^+\!\!L)^{1+v}$ scale, or merely reproduce known results in an equivalent form.

\subsection{Main results}

%To state our result, we recall the definition of a space of homogeneous type introduced by Coifman and Weiss \cite{CW71}.
%A nonzero measure $\mu$ is said to satisfy the \emph{doubling condition} if there exists a constant $C_{\mu}>0$, depending only on $\mu$, such that
%\[
%\mu(B(x,2r)) \le C_{\mu}\, \mu(B(x,r))
%\quad \text{for all } x\in X,\ r>0.
%\]
%
%We call $(X,d,\mu)$ a \emph{space of homogeneous type} if $(X,d,\mu)$ is a quasi-metric measure space and $\mu$ satisfies the doubling condition.
%We further assume that
%there exist constants $c,s>0$ such that
%\begin{equation}\label{upc}
%\mu(B(x,r))\le c r^s
%\end{equation}
%for all balls $B(x,r)$.
%This polynomial upper growth condition is weaker than Ahlfors regularity and is satisfied in many standard settings.
%In particular, it implies that balls whose measures are bounded below must have radius bounded below.%问题→洞察→结果→新问题→最终结论

Our first main result is stated as follows.
\begin{theorem}\label{thmfan}
Let $(X,d,\mu)$ be an $s$-Ahlfors-regular quasi-metric measure space.
Then
\begin{equation}\label{t1}
B^{0,v}_q(X) \subsetneq L(\log^+\!\!L)^{1+v}(X),
\end{equation}
for all $1<q\le \infty$ and $v>-1$.
\end{theorem}

%\begin{remark}
%{\it The nonatomic assumption guarantees that the Christ dyadic cubes can be chosen to have arbitrarily small measure (see \eqref{nonatomic}).
%Condition \eqref{upc} is satisfied by all Ahlfors-regular spaces,
%including $\mathbb R^n$, the unit sphere $\mathbb S^{n-1}$,
%and homogeneous groups.}
%%The assumption \eqref{upc} ensures that balls with measure bounded below also have a positive lower bound on their radius.
%\end{remark}

%\begin{remark}
%For a fixed $v$, the inclusion in \eqref{cor} allows one to transfer
%boundedness results for singular integral operators with kernels in
%$L(\log^+\!\!L)^{1+v}(\mathbb{S}^{n-1})$
%to the setting of kernels in $B^{0,v}_q(\mathbb{S}^{n-1})$.
%\end{remark}

%Theorem \ref{thmfan} establishes the inclusion
%$
%B^{0,v}_{q}(X)\subset L(\log^+L)^{1+v}(X).
%$

Theorem \ref{thmfan} demonstrates that the space $B^{0,v}_{q}(X)$ is a proper subspace of $L(\log^+\!\! L)^{1+v}(X)$.
A natural question is whether the space  $L(\log^+\!\! L)^{1+v}(X)$ appearing in \eqref{t1} can be replaced by a slightly smaller Orlicz-type space. The following theorem shows that this is impossible.

Let $v>-1$. For an increasing function $\Psi:[1,\infty)\to(0,\infty)$, define
\[
L(\log^+L)^{1+v}\Psi(L)(X)
=
\Big\{
f\in L^1(X):\,
\int_X |f(x)|
\bigl(\log(2+|f(x)|)\bigr)^{1+v}
\Psi(|f(x)|)\,d\mu(x)
<\infty
\Big\}.
\]

\begin{theorem}\label{3log}
Let $(X,d,\mu)$ be an $s$-Ahlfors-regular quasi-metric measure space.
Assume that $\Psi:[1, \infty)\to (0,\infty)$ is increasing and satisfies
$\lim_{t\to\infty}\Psi(t)=\infty$.
Then
\[
B^{0,v}_{q}(X)
\nsubseteq
L(\log^+\!\!  L)^{1+v}\Psi( L)(X),
\]
 for all $1<q\leq \infty$ and $v>-1$.
\end{theorem}

\begin{remark}
{\it Together, Theorems
\ref{thmfan}
and \ref{3log} show that $L(\log^+\!\! L)^{1+v}(X)$ is  the smallest Orlicz-type space containing $B^{0,v}_{q}(X)$ for all $1<q\leq \infty$ and $v>-1$. Consequently, the inclusion relationship in \eqref{t1} is optimal. }
\end{remark}

\begin{remark}
{\it When $X=\mathbb{S}^{n-1}$, Theorems \ref{thmfan} and  \ref{3log} yield the sharp inclusion relationship:
\begin{equation}\label{cor1}
B^{0,v}_{q}(\mathbb{S}^{n-1})\subsetneq L(\log^+\!\!L)^{1+v}(\mathbb{S}^{n-1})
\end{equation} and \begin{equation}\label{cor2}
B^{0,v}_{q}(\mathbb{S}^{n-1})
\nsubseteq
L(\log^+\!\! L)^{1+v}\Psi( L)(\mathbb{S}^{n-1}),
\end{equation}
thereby giving a confirmed  answer to Problem \ref{pro} and showing that $L(\log^+\!\!L)^{1+v}(\mathbb{S}^{n-1})$ is optimal among all Orlicz-type spaces containing $B^{0,v}_{q}(\mathbb{S}^{n-1})$.}
\end{remark}

%\begin{remark}
%{\it It should note that 
%the nonatomic assumption in Theorem \eqref{3log} cannot be removed. Indeed, if $X$ consists of finitely many atoms, then every measurable function on $X$ is in $L^{\infty}(X)$. Consequently,
%\[
%L(\log^+\!\! L)^{1+v}\Psi( L)(X)=L^\infty(X)=B^{0,v}_{q}(X).
%\]
%Hence the strict noninclusion in Theorem \eqref{3log} fails in this case. }
%\end{remark}

 However, if we remove 
the geometric  restriction requiring each block to be supported on a ball and instead introduce a generalized block space
$\mathscr B_q^{0,v}(X)$ that  depends only on the measure structure,  it might coincide exactly with Orlicz-type spaces $ L(\log^+\!\!L)^{1+v}(X)$.

Let $1<q\le\infty$.
A measurable function $a$ on $X$ is called a \emph{generalized $q$-block} if there exists a measurable set
$E\subset X$
with
$
0<\mu(E)<\infty
$
such that
$$
\supp a\subseteq E, \quad
\|a\|_{L^q(E)}
\le
\mu(E)^{1/q-1}.
$$

\begin{definition}[Generalized block space $\mathscr{B}^{k,v}_q(X)$]
{\it Let $1<q\le \infty$, $k\ge0$, and $v>-1$.
The space $\mathscr{B}^{k,v}_q(X)$ consists of all functions
$f\in L^1(X)$ admitting a decomposition
\[
f=\sum_{j=1}^\infty \lambda_j a_j,
\]
where $\lambda_j\in\mathbb C$ and each $a_j$ is a generalized $q$-block supported on a measurable set
$E_j\subset X$ with
\[
\sum_{j=1}^\infty |\lambda_j| \Bigl( 1+\phi_{k,v}\big(\mu(E_j)\big) \Bigr) <\infty.
\]}
\end{definition}

The space $\mathscr{B}^{k,v}_q(X)$ is equipped with the norm
\[
\|f\|_{\mathscr{B}^{k,v}_q(X)} := \inf \Big\{ \sum_{j=1}^\infty |\lambda_j| \Bigl( 1+\phi_{k,v}\big(\mu(E_j)\big) \Bigr) \Big\},
\]
where the infimum is taken over all the decompositions
$
f=\sum_{j=1}^\infty \lambda_j a_j
$
such that each $a_j$ is a generalized $q$-block supported on a measurable set
$E_j$.
It follows that $\mathscr{B}^{k,v}_q(X)$
is a Banach space.

The essential difference between
\(B^{k,v}_q(X)\)
and
\(\mathscr{B}^{k,v}_q(X)\)
is that the former requires each block to be supported on a ball, while the latter allows the supports of blocks to be arbitrary finite measurable sets in $X$.
Observe that every $q$-block supported on a ball is also a
generalized $q$-block.
Then we  immediately obtain
\begin{equation}\label{contain}
B^{k,v}_q(X)
\subseteq
\mathscr{B}_q^{k,v}(X).
\end{equation}

Now, the relationship between $\mathscr{B}^{0,v}_q(X) $ and $ L(\log^+\!\!L)^{1+v}(X)$ can be   formulated as follows.

\begin{theorem}\label{thm}
Let $(X,d,\mu)$ be an $s$-Ahlfors-regular quasi-metric measure space, and let
$1<q\le \infty$ and $v>-1$.
Then
\begin{equation}\label{key1}
\mathscr{B}^{0,v}_q(X) \subseteq L(\log^+\!\!L)^{1+v}(X).
\end{equation}

Moreover, if $\mu(X)<\infty$, then
\begin{equation}\label{key2}
\mathscr{B}^{0,v}_q(X) = L(\log^+\!\!L)^{1+v}(X).
\end{equation}
\end{theorem}

\begin{remark}{\it
Theorem \ref{thm} provides a characterization of the Orlicz-type space
$L(\log^+\!\!L)^{1+v}(X)$
in terms of generalized block-type decompositions in the case of $\mu(X)<\infty$.}
\end{remark}

%Theorem \eqref{thm} depends only on the measure structure and does not involve the quasi-metric $d$.
%Consequently, the inclusion \eqref{key1} remains valid for any quasi-metric measure space $(X,d,\mu)$.
%From \eqref{contain} and \eqref{e1} we immediately obtain
%\[
%B^{0,v}_q(X)
%\subseteq
%L(\log^+\!\!L)^{1+v}(X).
%\]
%for every quasi-metric measure space $(X,d,\mu)$.
%

\subsection{ Special Cases}
As a complement to the aforementioned results,
we consider two special cases that further reveal the role of the parameter $q$ and the logarithmic weight. The first is the endpoint case when  $q=1$. The second arises when the logarithmic weight is removed.

For the first case, we have the following result.
\begin{theorem}\label{B1case}
Let $(X,d,\mu)$ be an $s$-Ahlfors-regular quasi-metric measure space.
Then
\[
B^{0,v}_{1}(X)=L^1(X),
\]
for all \(v>-1\).
More precisely, for any \(f\in B^{0,v}_{1}(X)\),
\[
\|f\|_{L^1(X)}
\le
\|f\|_{B^{0,v}_1(X)}
\leq 
\|f\|_{L^1(X)} \Bigl( 1+ \frac1{1+v} \big(\log^+(\frac{1}{\mu(X)})\big)^{1+v} \Bigr).
\]
\end{theorem}

The second  case concerns the block spaces introduced by
Blasco, Ruiz and Vega \cite{Bla99} (see also \cite{Dun24}), which serve as the preduals of Morrey spaces.
We say that $a(x)$ is a \emph{$(q, \alpha)$-block} with $ \alpha \geq 0$ if there exists a ball $B(x,t) \subseteq X$  such that
$$
\supp a \subseteq B(x,t),\quad 
\|a\|_{L^q(B(x,t))}\leq \mu(B(x,t))^{-1/q'} t^{\alpha}.
$$

\begin{definition}[Block space $\tilde{B}_{q,\alpha}(X)$]{\it
Let $1<q\le \infty$ and $\alpha \geq 0$.
The space $\tilde{B}_{q,\alpha}(X)$ consists of all functions $f\in L^1(X)$ admitting a decomposition
\[
f = \sum_{j=1}^{\infty} \lambda_j a_j,
\]
where each $a_j$ is a $(q,\alpha)$-block supported on a ball $B_j\subseteq X$,
$\lambda_j\in\mathbb{C}$, and
$
\sum_{j=1}^{\infty} |\lambda_j|<\infty.
$}
\end{definition}

The norm on $\tilde{B}_{q,\alpha}(X)$ is given by
\[
\|f\|_{\tilde{B}_{q,\alpha}(X)}
:=\inf\Big\{\sum_{j=1}^{\infty}|\lambda_j|\Big\},
\]
where the infimum is taken over all block decompositions
$
f=\sum_{j=1}^{\infty}\lambda_j a_j
$
such that $\lambda_j\in\mathbb C$ and each $a_j$ is a $(q,\alpha)$-block supported on a ball $B_j$.

It is easy to see that
$
(\tilde{B}_{q,\alpha}(X),\|\cdot\|_{\tilde{B}_{q,\alpha}(X)})
$
is a Banach space.
The case when $\alpha=0$ plays a special role in this paper.
Indeed,
$\tilde B_{q,0}(X)$
can be regarded as the block space
$B_q^{0,v}(X)$
without the logarithmic weight. Moreover, we observe that the space $\tilde B_{q,0}(X)$ also turns out to be $L^1(X)$. The result can be presented as follows.

%Recently,
%Dung, Dao, Duong and Nghia \cite{Dun24} showed that $B^{0,v}_q(X)$ is the dual of $M_{q',\alpha}(X)$ for $\alpha>0$,
%where the Morrey space
%$M_{q,\alpha}(X)$
%consists of all measurable functions $f$ satisfying
%$$
%\|f\|_{M_{q,\alpha}(X)}= \sup_{B(x,t)} t^{\alpha} \Big( \frac{1}{ \mu(B(x,t)) } \int_{B(x,t)}|f(y)|^q  d \mu(y)  \Big)^{1/q}<\infty.
%$$
%If $(X,d,\mu)$ is a space of homogeneous type with a regular doubling Borel measure, then the Lebesgue differentiation theorem (see, e.g., \cite{CW71, Tol04}) holds. 
%Consequently, when $\alpha=0$, the Morrey space $M_{q,0}(X)$ coincides with $L^\infty(X)$.
%This naturally raises the question of whether the associated block space reduces to $L^1(X)$.
%The following theorem gives an affirmative answer.

\begin{theorem}\label{Bqcase}
Let $(X,d,\mu)$ be an $s$-Ahlfors-regular quasi-metric measure space.
Then
$$
\tilde{B}_{q,0}(X)=L^1(X),
$$
for all $1<q\le \infty$.
In particular, $$ \|f\|_{\tilde{B}_{q,0}(X)}\approx \|f\|_{L^1(X)}. $$
\end{theorem}

\begin{remark}{\it Theorems \ref{B1case} and \ref{Bqcase} show that the corresponding block spaces coincide with $L^1(X)$ either at the endpoint case $q=1$
or when the logarithmic weight is removed.}\end{remark}

This paper is organized as follows.
In Section \ref{sec3}, we establish the proper inclusion
$B_q^{0,v}(X)\subsetneq L(\log^+\!\!L)^{1+v}(X)$
and prove Theorem \ref{thmfan}.
Section \ref{sec4} is devoted to study the sharpness of this inclusion via proving Theorem \ref{3log}.
In Section \ref{sec2}, we study the generalized block space
$\mathscr B_q^{0,v}(X)$
and establish Theorem \ref{thm}.
Finally, Sections \ref{sec5} and \ref{sec6} provide the proofs of Theorems \ref{B1case} and \ref{Bqcase}, respectively.

\medskip
\noindent
\textbf{Notation.}
For $p\in(1,\infty)$, let $p'=p/(p-1)$ denote the conjugate exponent, with the conventions $1'=\infty$ and $\infty'=1$. 
We use $C$ and $c$ to denote positive constants which may change from line to line. 
We write $A\lesssim B$ if $A\leq CB$, and $A\approx B$ if $A\lesssim B$ and $A\lesssim B$.

\section{Proof of Theorem \ref{thmfan}}\label{sec3}
\setcounter{equation}{0}%公式从1开始编号

To prove Theorem \ref{thmfan}, it suffices to prove  the inclusion relation
\begin{equation}\label{e1}
B_q^{0,v}(X)
\subseteq
L(\log^+\!\!L)^{1+v}(X)
\end{equation}
and to construct a function
\begin{equation}\label{example}
f\in L(\log^+\!\!L)^{1+v}(X)
\setminus
B_q^{0,v}(X).
\end{equation}

In the following, we will prove \eqref{e1} and \eqref{example}, respectively.
\subsection{Proof of (\ref{e1})}
By the definition of ${B}^{0,v}_q(X)$, every $q$-block is also a $p$-block for $1<p<q\le\infty$. Consequently,
\begin{equation}\label{Bbh}
{B}^{0,v}_q(X) \subseteq {B}^{0,v}_p(X),\quad 1<p<q\le \infty.
\end{equation}

Thus, it remains to consider the case of $1<q< \infty.$
Suppose $f\in {B}^{0,v}_q(X)$.
Then there exists a decomposition
\[
f=\sum_{j=1}^{\infty}\lambda_j a_j,
\]
where $\lambda_j\in \mathbb C$ and each $a_j$ is a $q$-block satisfying
\begin{equation}\label{al}
\supp a_j \subseteq B_j\subseteq X, \quad \|a_j\|_{L^q(B_j)} \le \mu(B_j)^{-1/q'}.
\end{equation}
By H\"older's inequality,
\begin{equation}\label{al1}
\|a_j\|_{L^1(X)} \le \mu(B_j)^{1/q'} \|a_j\|_{L^q(B_j)} \le 1.
\end{equation}
Moreover,
\[
\|f\|_{{B}^{0,v}_q(X)} \le \sum_{j=1}^{\infty} |\lambda_j| \bigl(1+\phi_{0,v}(\mu(B_j))\bigr)<\infty.
\]

Since
\begin{equation}\label{phi0v}
\phi_{0,v}(t)=\frac{1}{1+v}\bigl(\log^+(\frac{1}{t})\bigr)^{1+v},
\end{equation}
we obtain
\[\sum_{j=1}^{\infty}|\lambda_j|\Big(1+\frac1{1+v}\bigl(\log^+(\frac{1}{\mu(B_j)})\bigr)^{1+v}\Big)<\infty.
\]

To prove \eqref{e1}, it remains to show that
$$\|f\|_{L(\log^+\!\!L)^{1+v}(X)}<\infty.$$
We now show that the problem can be reduced to the case
$\sum_{j=1}^{\infty}|\lambda_j|=1.$
Set
$\alpha=\sum_{j=1}^{\infty}|\lambda_j|.$
If $\alpha\neq 1$, we define
\[
\tilde{\lambda}_j=\alpha^{-1}\lambda_j
\quad\text{and}\quad
\tilde{f}=\alpha^{-1}f.
\]

To justify this reduction, we will use the following elementary inequalities.
First, for all $a,b>0$,
\begin{equation}\label{logab}
\log^+(ab) \le \log^+a+\log^+b.
\end{equation}

Next, for all $a>0$,
\begin{equation}\label{log2+a}
\log^+a \le \log(2+a) \le \log4+\log^+a.
\end{equation}

Consequently, for all $a,b>0$,
\begin{equation}\label{logab2}
\log(2+ab) \le \log4+\log^+a+\log^+b.
\end{equation}

Finally, for all $a,b,r>0$,
\begin{equation}\label{abr}
(a+b)^r \le 2^r(a^r+b^r).
\end{equation}

By \eqref{log2+a}, \eqref{logab2}, and \eqref{abr}, we have
\begin{equation}\label{rescal}
\begin{aligned}
\|f\|_{L(\log^+\!\!L)^{1+v}(X)}
&=\int_X |\alpha \tilde f| \bigl(\log(2+|\alpha\tilde f|)\bigr)^{1+v} \,d\mu\\
&\le \alpha \int_X |\tilde f|\Bigl(\log4+\log^+\alpha+\log^+|\tilde f|\Bigr)^{1+v}\,d\mu
\\
&\lesssim \alpha(1+\log^+\alpha)^{1+v}\|\tilde f\|_{L^1(X)}
+ \alpha \|\tilde f\|_{L(\log^+\!\!L)^{1+v}(X)}.
\end{aligned}
\end{equation}
Using the fact that
$
\|\tilde f\|_{L^1(X)} \lesssim \|\tilde f\|_{L(\log^+\!\!L)^{1+v}(X)},
$
we further obtain
\[
\|f\|_{L(\log^+\!\!L)^{1+v}(X)} \lesssim \alpha \bigl(1+(\log^+\alpha)^{1+v}\bigr) \|\tilde f\|_{L(\log^+\!\!L)^{1+v}(X)}.
\]

On the other hand,
\begin{align*}
\|\tilde f\|_{{B}^{0,v}_q(X)}
&\le \sum_{j=1}^{\infty} |\tilde{\lambda}_j| \Big(1+\frac1{1+v} \bigl(\log^+(\frac{1}{\mu(B_j)})\bigr)^{1+v}\Big)\\
&=\alpha^{-1} \sum_{j=1}^{\infty} |\lambda_j|\Big(1+\frac1{1+v}\bigl(\log^+(\frac{1}{\mu(B_j)})\bigr)^{1+v}\Big)<\infty.
\end{align*}
This means that it suffices to prove $\|\tilde f\|_{L(\log^+\!\!L)^{1+v}(X)}<\infty$ under condition $\|\tilde f\|_{{B}^{0,v}_q(X)}<\infty$.
Therefore, we can  assume that
$
\sum_{j=1}^{\infty} |\lambda_j| =1.
$

We now estimate the
$L(\log^+\!\!L)^{1+v}(X)$
quasi-norm of $f$ under the assumptions $\sum_{j=1}^{\infty} |\lambda_j|=1$ and
\[
\sum_{j=1}^{\infty} |\lambda_j| \Big(1+\frac1{1+v}\bigl(\log^+(\frac{1}{\mu(B_j)})\bigr)^{1+v}\Big)<\infty.
\]

Since the function
$
t\mapsto t(\log(2+t))^{1+v}
$
is convex, Jensen's inequality yields
\begin{align*}
\|f\|_{L(\log^+\!\!L)^{1+v}(X)}
&=\Big\|\sum_{j=1}^{\infty} \lambda_j a_j \Big\|_{L(\log^+\!\!L)^{1+v}(X)} \le \sum_{j=1}^{\infty} |\lambda_j|\,\|a_j\|_{L(\log^+\!\!L)^{1+v}(B_j)}.
\end{align*}
Hence, it remains to prove that
\begin{equation}\label{claim}
\|a_j\|_{L(\log^+\!\!L)^{1+v}(B_j)} \lesssim 1+\frac1{1+v} \bigl(\log^+(\frac{1}{\mu(B_j)})\bigr)^{1+v}.
\end{equation}

Indeed, once \eqref{claim} is established, we obtain
\begin{align*}
\|f\|_{L(\log^+\!\!L)^{1+v}(X)}
&\le \sum_{j=1}^{\infty} |\lambda_j|\, \|a_j\|_{L(\log^+\!\!L)^{1+v}(B_j)}\\
&\lesssim \sum_{j=1}^{\infty} |\lambda_j| \Big(1+\frac1{1+v}\bigl(\log^+(\frac{1}{\mu(B_j)})\bigr)^{1+v}\Big)<\infty.
\end{align*}
Hence,
$
f\in L(\log^+\!\!L)^{1+v}(X).
$
This completes the proof of \eqref{e1}, modulo the verification of \eqref{claim}.

Before proving \eqref{claim}, we derive an estimate for the
$L(\log^+\!\!L)^{1+v}(B_j)$
quasi-norm of a function $g$.
Using \eqref{log2+a} and \eqref{abr}, for any
$g\in L(\log^+\!\!L)^{1+v}(B_j)$,
we have
\begin{equation}\label{logq0}
\begin{aligned}
\|g\|_{L(\log^+\!\!L)^{1+v}(B_j)}
&=\int_{B_j} |g| (\log(2+|g|))^{1+v}\,d\mu\\
&\le \int_{B_j} |g| (\log4+\log^+|g|)^{1+v} \,d\mu\\
&\le \int_{B_j} |g| 2^{1+v} \Big((\log4)^{1+v}+(\log^+|g|)^{1+v} \Big)\,d\mu\\
&\lesssim \int_{B_j}|g|\,d\mu + \int_{B_j} |g|(\log^+|g|)^{1+v}\,d\mu.
\end{aligned}
\end{equation}

For the second term, we apply the elementary inequality
\[
\log t\le \frac{t^\delta}{\delta}, \qquad t>1,\ \delta>0.
\]
This yields
\begin{equation}\label{logq}
\begin{aligned}
\int_{B_j} |g| (\log^+|g|)^{1+v}\,d\mu
&= \int_{\{x\in B_j:|g(x)|>1\}} |g|(\log|g|)^{1+v}\,d\mu\\
&\le \int_{\{x\in B_j:|g(x)|>1\}} |g| \Big(\frac{|g|^\delta}{\delta}\Big)^{1+v}\,d\mu\\
&\le\Big(\frac1\delta\Big)^{1+v}\int_{B_j} |g|^{\delta(1+v)+1} \,d\mu\\
&=\frac{\|g\|_{L^{\delta(1+v)+1}(B_j)}^{\delta(1+v)+1}}{\delta^{1+v}}.
\end{aligned}
\end{equation}

Combining \eqref{logq0} and \eqref{logq}, we obtain
\begin{equation}\label{logq1}
\begin{aligned}
\|g\|_{L(\log^+\!\!L)^{1+v}(B_j)}
&\lesssim \|g\|_{L^1(B_j)} + \frac{\|g\|_{L^{\delta(1+v)+1}(B_j)}^{\delta(1+v)+1}}{\delta^{1+v}},
\end{aligned}
\end{equation}
where $\delta>0$ will be chosen later.

We now turn to the proof of \eqref{claim}.
Taking
$
\delta=\frac{q-1}{1+v}
$
and applying \eqref{logq1} to $a_j$ together with \eqref{al}, \eqref{al1}, we obtain
\begin{align*}
\|a_j\|_{L(\log^+\!\!L)^{1+v}(B_j)}
\lesssim \|a_j\|_{L^1(B_j)} + \|a_j\|_{L^q(B_j)}^q
\lesssim 1+\mu(B_j)^{-q/q'}.
\end{align*}
However, this bound is not sufficient to control the right-hand side of \eqref{claim}.

To overcome this difficulty, we introduce a rescaling.
Define
\[
\tilde a_j=\gamma_j a_j, \qquad \gamma_j>0,
\]
where $\gamma_j$ will be chosen later.
By \eqref{al}, we have
\begin{equation}\label{aw1}
\begin{aligned}
\|\tilde a_j\|_{L^q(B_j)}
&=\gamma_j \|a_j\|_{L^q(B_j)} \le \gamma_j \mu(B_j)^{-1/q'},\\ 
\|\tilde a_j\|_{L^1(B_j)}
&= \gamma_j \|a_j\|_{L^1(B_j)} \le \gamma_j.
\end{aligned}
\end{equation}

Applying \eqref{logq1} to $\tilde a_j$ with
$
\delta=\frac{q-1}{1+v},
$
and using \eqref{aw1}, we obtain
\begin{equation}\label{aw2}
\begin{aligned}
\|\tilde a_j\|_{L(\log^+\!\!L)^{1+v}(B_j)}
&\lesssim \|\tilde a_j\|_{L^1(B_j)} + \|\tilde a_j\|_{L^q(B_j)}^q \lesssim \gamma_j +\gamma_j^q \mu(B_j)^{-q/q'}.
\end{aligned}
\end{equation}

By the same argument of \eqref{rescal}, we have
\begin{align*}
\|a_j\|_{L(\log^+\!\!L)^{1+v}(B_j)}
\lesssim \frac1{\gamma_j} \Big(1+\log^+(\frac1{\gamma_j})\Big)^{1+v} \|\tilde a_j\|_{L^1(B_j)}
+ \frac1{\gamma_j} \|\tilde a_j\|_{L(\log^+\!\!L)^{1+v}(B_j)}.
\end{align*}
Applying \eqref{aw1} and \eqref{aw2}, we have
\begin{align*}
\|a_j\|_{L(\log^+\!\!L)^{1+v}(B_j)}
\lesssim 1 + (\log^+(\frac{1}{\gamma_j}))^{1+v}+ \gamma_j^{q-1} \mu(B_j)^{-q/q'}.
\end{align*}

Thus, to complete the proof of \eqref{claim}, it remains to choose $\gamma_j$ such that
\begin{equation}\label{aw3}
1 + \bigl( \log^+(\frac{1}{\gamma_j}) \bigr)^{1+v}+ \gamma_j^{q-1} \mu(B_j)^{-q/q'}
\lesssim 1+ \frac1{1+v} \bigl( \log^+(\frac{1}{\mu(B_j)}) \bigr)^{1+v},
\end{equation}
for any $1<q<\infty$ and $v>-1$.

We claim that \eqref{aw3} can be achieved by an appropriate choice of $\gamma_j$.
Indeed, suppose that there exists a constant $\eta>0$ such that
\[
\gamma_j^{q-1}
\mu(B_j)^{-q/q'}
\le
\eta.
\]
This immediately yields an upper bound for $\gamma_j$:
\begin{equation}\label{upb}
\gamma_j \le \eta^{\frac1{q-1}}\mu(B_j).
\end{equation}

Next, we consider the logarithmic term.
Suppose that there exists a constant $\kappa>0$ such that
\[
\bigl(\log^+(\frac1{\gamma_j})\bigr)^{1+v} 
\le\kappa\frac1{1+v} \bigl( \log^+(\frac1{\mu(B_j)})\bigr)^{1+v}.
\]
This immediately yields a lower bound for $\gamma_j$:
\begin{equation}\label{lowb}
\gamma_j \ge\mu(B_j)^{\big(\frac{\kappa}{1+v}\big)^{\frac1{1+v}}}.
\end{equation}

Combining \eqref{upb} and \eqref{lowb}, we obtain
\[
\mu(B_j)^{\big( \frac{\kappa}{1+v} \big)^{\frac1{1+v}}}
\le \gamma_j
\le \eta^{\frac1{q-1}} \mu(B_j).
\]

Choosing
$
\kappa=1+v
$
and 
$
\eta=1,
$
we may take
$
\gamma_j=\mu(B_j).
$
This choice works uniformly for all $1<q<\infty$ and $v>-1$, and hence \eqref{aw3} holds.
The proof of \eqref{e1} is complete.

\subsection{Proof of (\ref{example})}

We first establish some useful lemmas which play an important role in the proof of (\ref{example}).
\begin{lemma}\label{jz}
Let $1<q\le \infty$ and let $g$ be a measurable function.
Suppose that there exists a constant $C_g>0$ such that for every ball $B\subseteq X$,
\begin{equation}\label{jzy}
\Big(
\frac{1}{\mu(B)}
\int_B |g(x)|^{q'}\,d\mu(x)
\Big)^{1/q'}
\le
C_g
\Big(
1+\frac{1}{1+v}
\big(\log^+ \frac{1}{\mu(B)}\big)^{1+v}
\Big).
\end{equation}
Then for every $f\in B^{0,v}_q(X)$, we have
\[
\|fg\|_{L^1(X)}
\le
C_g\,\|f\|_{B^{0,v}_q(X)}.
\]
\end{lemma}

\begin{proof}
Let $f\in B^{0,v}_{q}(X)$. 
Then there exists a decomposition 
\[
f=\sum_{j=1}^{\infty}\lambda_j a_j,
\]
where $\lambda_j\in\mathbb C$ and each $a_j$ is a $q$-block satisfying
\[
\supp a_j\subseteq B_j\subseteq X,
\quad
\|a_j\|_{L^q(B_j)}\le \mu(B_j)^{-1/q'},
\]
and
\[
\sum_{j=1}^{\infty} |\lambda_j| \Big( 1+ \frac1{1+v} \bigl(\log^+(\frac1{\mu(B_j)})\bigr)^{1+v}  \Big) <\infty.
\]
Applying \eqref{jzy}, we obtain
\begin{align*}
\|fg\|_{L^1(X)} &=\|\sum_{j=1}^{\infty}\lambda_j a_j g\|_{L^1(X)}\leq \sum_{j=1}^{\infty}|\lambda_j|  \|a_j g\|_{L^1(X)}\\
&\leq \sum_{j=1}^{\infty}|\lambda_j|  \|a_j\|_{L^q(B_j)} \|g\|_{L^{q'}(B_j)} \\
&\leq C_g \sum_{j=1}^{\infty}|\lambda_j| \mu(B_j)^{-1/q'} \mu(B_j)^{1/q'} \Big( 1+\frac{1}{1+v}\bigl(\log^+(\frac1{\mu(B_j)})\bigr)^{1+v} \Big)\\
&= C_g \sum_{j=1}^{\infty}|\lambda_j| \Big( 1+ \frac1{1+v} \bigl(\log^+(\frac1{\mu(B_j)})\bigr)^{1+v}  \Big).
\end{align*}

Recalling the definition of $\|\cdot\|_{B^{0,v}_{q}(X)}$, we obtain
\begin{align*}
\|fg\|_{L^1(X)} \leq C_g \|f\|_{B^{0,v}_{q}(X)}.
\end{align*}
This completes the proof of Lemma \ref{jz}.
\end{proof}

\begin{lemma}\label{Gk}
Let $(X,d,\mu)$ be an $s$-Ahlfors-regular quasi-metric measure space with $\mu(X)\geq 1$.
Then there exist pairwise disjoint measurable sets
\[ G_1,G_2,\ldots \]
such that
\[ 2^{-(k+1)} < \mu(G_k) \le 2^{-k}, \qquad k\ge1, \]
and each $G_k$ is a union of Christ cubes.
\end{lemma}

\begin{proof}
Since  $(X,d,\mu)$ is an $s$-Ahlfors-regular quasi-metric measure space, it is nonatomic and admits a Christ dyadic system (see \cite{Chr90} and \cite{Hyt12})
$$ 
\mathscr D = \{Q_\alpha^n:\alpha\in I_n,\ n\in\mathbb Z\}.
$$
Therefore, every descending chain of Christ cubes has vanishing measure, which implies
\begin{equation}\label{nonatomic}
\sup_{\alpha\in I_n} \mu(Q_\alpha^n) \longrightarrow 0 \qquad \text{as} \quad n\to\infty.
\end{equation}

To prove Lemma \ref{Gk}, we take an increasing sequence of generations
$
n_1<n_2<\cdots
$
such that
\[
\sup_{\alpha\in I_{n_k}} \mu(Q_\alpha^{n_k}) \le 2^{-(k+1)} \qquad k\ge1.
\]

For $k=1$, 
enumerate the family
\[
\{Q_\alpha^{n_1}\}_{I_{n_1}}
=
\{Q_1^{n_1},Q_2^{n_1},\dots\}.
\]

Since $\mu(X)\geq 1$, we start with the empty family and successively select cubes, ensuring that after each addition, the total measure of the selected cubes remains at most  \(2^{-1}\).
Let \(\mathcal F_{1}\) denote the resulting family.
By construction,
\[
\sum_{Q\in\mathcal F_{1}}\mu(Q)
\le 2^{-1}.
\]
Moreover, the family is maximal with respect to inclusion.
Set
\[
G_1 = \bigcup_{Q\in\mathcal F_1}Q.
\]
Since every cube in
$\{Q_\alpha^{n_1}\}$
has measure at most $2^{-2}$,
the maximality of
$\mathcal F_1$
implies
$$
2^{-2} < \mu(G_1) \le 2^{-1}.
$$

Suppose that
$
G_1,\ldots,G_k
$
have been constructed.
Let
\[
R_k = X\setminus\bigcup_{j=1}^{k}G_j .
\]

Since $\mu(X)\geq 1$, we have
\[
\mu(R_k)=\mu(X)-\sum_{j=1}^{k}\mu(G_j) \geq 1-\sum_{j=1}^{k}2^{-j} = 2^{-k}.
\]

We now select $G_{k+1}$ from $R_k$.
Observe that $R_k$ is a union of Christ cubes.
For every Christ cube $Q\subset R_k$ with
$
\mu(Q)>2^{-(k+2)},
$
we replace $Q$ by all of its descendants at generation $n_{k+1}$.
Since
\[
\sup_{\alpha\in I_{n_{k+1}}}\mu(Q_\alpha^{n_{k+1}})
\le 2^{-(k+2)},
\]
this procedure yields a pairwise disjoint family of Christ cubes
$\{Q_{\beta_k}\}_{\beta_k}$ such that
\[
 R_k =\bigcup_{\beta_k} Q_{\beta_k}  
\]
and
\[
\mu(Q_{\beta_k}) \le 2^{-(k+2)}.
\]

Enumerate the family
\[
\{Q_{\beta_k}\}_{\beta_k}
=
\{Q_1,Q_2,\dots\}.
\]

Starting from the empty family,
select cubes successively whenever adding the cube keeps the total measure
at most \(2^{-(k+1)}\).
Let \(\mathcal F_{k+1}\) denote the resulting family.
By construction,
\[
\sum_{Q\in\mathcal F_{k+1}}\mu(Q)
\le 2^{-(k+1)}.
\]
Moreover, the family is maximal with respect to inclusion.
Set
\[
G_{k+1} = \bigcup_{Q\in\mathcal F_{k+1}}Q.
\]
Since every cube in
$\{Q_{\beta_k}\}_{\beta_k}$
has measure at most
$2^{-(k+2)}$,
the maximality of
$\mathcal F_{k+1}$
implies
\[
2^{-(k+2)} < \mu(G_{k+1}) \le 2^{-(k+1)}.
\]

Moreover, by $ G_{k+1}\subset R_k,$ we have
\[
G_{k+1}\cap G_j=\varnothing, \qquad 1\le j\le k.
\]
This completes the proof of Lemma \ref{Gk}.
\end{proof}
\begin{lemma}\label{lemEm}
Let $(X,d,\mu)$ be an $s$-Ahlfors-regular quasi-metric measure space with $\mu(X)\geq 1$.
Then there exists an absolute constant $C_0>0$ such that for every $m\ge 5$, 
there is a measurable set $E_m \subset X$ satisfying
\[
\mu(E_m) = 2^{-2m-2}, \qquad 
\mu(B\cap E_m) \le C_0 2^{-m} \mu(B)
\]
for every ball $B$ with 
$
\mu(B)\ge 2^{-m^2}.
$
Moreover, the sets $\{E_m\}_{m\ge5}$ may be chosen pairwise disjoint.
\end{lemma}

\begin{proof}
We use the same Christ dyadic system as in Lemma~\ref{Gk}:
\[
\mathscr D = \{ Q_\alpha^k : \alpha \in I_k, k \in \mathbb Z \}.
\]

By Lemma~\ref{Gk}, there exist pairwise disjoint sets
\[
G_5, G_6, \dots \subset X
\]
such that
\begin{equation}\label{Gm}
2^{-(m+1)} < \mu(G_m) \le 2^{-m}, \qquad m \ge 5,
\end{equation}
and each $G_m$ is a union of Christ cubes.

Fix $m$, since $(X,d,\mu)$ is a nonatomic space, we have
\[
\sup_{\alpha \in I_k} \mu(Q_\alpha^k) \to 0 \quad \text{as } k \to \infty.
\]
Let $B = B(x,r)$ with $\mu(B) \ge  2^{-m^2}$. 
Then by \eqref{upc}, we get $r\geq  (C^{-1}\mu(B))^{1/s}\geq  (C^{-1} 2^{-m^2})^{1/s} $.

Choose $N_m$ large enough so that
\[
\sup_{\alpha \in I_{N_m}} \mu(Q_\alpha^{N_m}) \le  2^{-m^2-4m},\qquad  \operatorname{diam}(Q_\alpha^{N_m})\leq (C^{-1} 2^{-m^2})^{1/s}/2.
\]

Refine all Christ cubes appearing in $G_m$ to generation $N_m$ to obtain a pairwise disjoint family
\[
\mathcal P_m = \{ P_\beta \}_\beta \subset G_m
\]
such that
\[
G_m = \bigcup_\beta P_\beta, \qquad \mu(P_\beta) \le  2^{-m^2-4m},\qquad  \operatorname{diam}(P_\beta)\leq (c^{-1} 2^{-m^2})^{1/s}/2.
\]

By nonatomicity, for each $P_\beta$, we can choose a measurable subset
$
F(P_\beta) \subset P_\beta
$
satisfying
\[
2^{-m-1} \mu(P_\beta) \le \mu(F(P_\beta)) \le 2^{-m} \mu(P_\beta).
\]

Define
$
F_m := \bigcup_\beta F(P_\beta).
$
Then
\[
2^{-m-1} \mu(G_m) \le \mu(F_m) \le 2^{-m} \mu(G_m).
\]
By \eqref{Gm}, we get 
\[
2^{-2m-2} < \mu(F_m) \le 2^{-2m}.
\]

Again by nonatomicity, we may select a measurable subset
\[
E_m \subset F_m
\]
with $\mu(E_m) = 2^{-2m-2}$.
Since the $G_m$'s are pairwise disjoint, so are the $E_m$'s.

Let 
\[
\mathcal P_m(B) := \{ P_\beta \in \mathcal P_m : P_\beta\cap B \neq \emptyset \}.
\]
Since each $P_\beta$ is a Christ cube has diameter much smaller than $r$, i.e. $\operatorname{diam}(P_\beta) \leq r/2$. 
Then, for any $P_\beta$ intersecting $B$, we have 
\[
P_\beta \subset B(x, r + \operatorname{diam}(P_\beta)) \subset B(x, 2 r).
\]
By \eqref{upc}, we obtain
\begin{align*}
\mu(F_m \cap B) 
&\leq \sum_{P_\beta \in \mathcal P_m(B) } \mu(F(P_\beta)) 
\le 2^{-m} \sum_{P_\beta \in \mathcal P_m(B) } \mu(P_\beta) \\
&\leq 2^{-m} \mu(C B) 
\le C_0 2^{-m} \mu(B),
\end{align*}
where $C_0>0$ is a constant.
This completes the proof of Lemma \ref{lemEm}.
\end{proof}

We now prove \eqref{example}. 
Without loss of generality, we assume $\mu(X)\geq 1$,
Indeed, if $\mu(X)<1$, we may rescale the measure by setting $d\tilde\mu = d\mu/\mu(X)$ so that $\mu(X)=1$, as the desired estimate is homogeneous under such scaling.

Let $m\geq 5$, $\lambda_m=\frac{1}{m^{2+v}(\log m)^2}$.
By Lemma \ref{lemEm}, there exists an absolute constant $C_0>0$ such that for every $m\ge 5$, 
there is a measurable set $E_m \subset X$ satisfying
\begin{equation}\label{BnEm}
\mu(E_m) = 2^{-2m-2}, \qquad 
\mu(B\cap E_m) \le C_0 2^{-m} \mu(B)
\end{equation}
for every ball $B$ with 
$
\mu(B)\ge 2^{-m^2}.
$ 
Moreover, the sets $\{E_m\}_{m\ge5}$ may be chosen pairwise disjoint.

Define
$$f=\sum_{m\geq 5} \frac{\lambda_m}{\mu(E_m)} \chi_{E_m}.$$
Then we have
\begin{align*}
\int_X |f| \bigl( \log(2+|f|) \bigr)^{1+v} \,d\mu
&= \sum_{m\geq 5}  \int_{E_m} \frac{|\lambda_m|}{\mu(E_m)} \biggl( \log\Bigl(2+ \frac{|\lambda_m|}{\mu(E_m)} \Bigr) \biggr)^{1+v}  \,d\mu
\\
&=\sum_{m\geq 5} \frac1{ m^{2+v} (\log m)^2}
\biggl( \log\Bigl( 2+
\frac{2^{2m+2}}{ m^{2+v} (\log m)^2}\Bigr) \biggr)^{1+v}\\
&\lesssim \sum_{m\geq 5} \frac{(m+1)^{1+v}}{ m^{2+v}  (\log m)^2}<\infty.
\end{align*}
Thus we get $f\in L(\log^+\!\!L)^{1+v}(X) $.

Define 
$$
g=\sum_{m\geq 5} m^{2(1+v)} \chi_{E_m}.
$$
Since $v>-1$, we have
$$
\|fg\|_{L^1(X)} =\sum_{m\geq 5} \frac{\lambda_m}{\mu(E_m)} m^{2(1+v)} \mu(E_m)= \sum_{m\geq 5} \frac{m^{v}}{(\log m)^2}=\infty.
$$
Hence, by Lemma \ref{jz}, to prove $f\notin B^{0,v}_{q}(X)$, it remains to show
$g$ satisfies \eqref{jzy}.

For $B\subset X$, if $\mu(B)\geq 1$, we have
\begin{align*}
\int_{B} |g(x)|^{q'}d\mu(x) 
&=\sum_{m\geq 5} m^{2(1+v)q'}\mu(B\cap E_m) 
\leq \sum_{m\geq 5} m^{2(1+v)q'}\mu( E_m) \\
&\leq  \sum_{m\geq 5} m^{2(1+v)q'}2^{-2m-2} 
\lesssim 1.
\end{align*}
Then we obtain
$$
\Big(\frac{1}{\mu(B)}\int_{B} |g(x)|^{q'}d\mu(x)  \Big)^{1/q'}\lesssim  1+ \frac1{1+v} \bigl(\log^+(\frac1{\mu(B)})\bigr)^{1+v}. 
$$

If $0<\mu(B)<1$,
we set $$\delta_{B,v} = 1+ \frac1{1+v} \bigl(\log(\frac1{\mu(B)})\bigr)^{1+v},\quad \gamma_{B} = 1+ \log_2(\frac{1}{\mu(B)}). $$ 
Then by \eqref{abr}, we have $(\gamma_{B})^{1+v}\lesssim \delta_{B,v}$.

Recalling that $
g=\sum_{m\geq 5} m^{2(1+v)} \chi_{E_m}
$, we split the resulting sum according to $\gamma_{B}$.
This yields
\begin{align*}
\int_{B} |g(x)|^{q'}d\mu(x) 
&=\sum_{m\geq 5} m^{2(1+v)q'}\mu(B\cap E_m) \\
&= \sum_{5\leq m \leq \sqrt{\gamma_{B}} } m^{2(1+v)q'}\mu(B\cap E_m)
+\sum_{ m\geq  \sqrt{\gamma_{B}}  } m^{2(1+v)q'}\mu(B\cap E_m)\\
&=:I_1+I_2.
\end{align*}

For $I_1$, since $\{E_m\}_{m\geq 5}$ are pairwise disjoint, we have 
\begin{align*}
I_1
 &= \sum_{ 5 \leq m \leq \sqrt{\gamma_{B}} } m^{2(1+v)q'} \mu(B\cap E_m)\\
 &\leq (\gamma_{B})^{(1+v)q'}  \sum_{5\leq m \leq \sqrt{\gamma_{B}} } \mu(B\cap E_m)\\
 &\lesssim (\delta_{B,v})^{q'}\mu(B).
\end{align*}

For $I_2$, since $ m \geq \sqrt{\gamma_{B}} $, we have
$
 \mu(B)\geq 2^{-m^2}.
$
By \eqref{BnEm}, we obtain
$$
\mu(B\cap E_m)\leq C_0 2^{-m} \mu(B).
$$
Then we have
\begin{align*}
I_2&=\sum_{ m\geq  \sqrt{\gamma_{B}}  } m^{2(1+v)q'}\mu(B\cap E_m)\\
&\leq C_0 \sum_{ m\geq  \sqrt{\gamma_{B}} } m^{2(1+v)q'}2^{-m} \mu(B)\\
&\lesssim \mu(B).
\end{align*}
It follows that
$$
\int_{B} |g(x)|^{q'}d\mu(x) \leq I_1+I_2 \lesssim (\delta_{B,v})^{q'}\mu(B).
$$
This implies 
$$
\Big(\frac{1}{\mu(B)}\int_{B} |g(x)|^{q'}d\mu(x)  \Big)^{1/q'}\lesssim   1+ \frac1{1+v} \bigl(\log^+(\frac1{\mu(B)})\bigr)^{1+v} . 
$$
Hence, $g$ is a function satisfying \eqref{jzy}.

Therefore $f\notin B^{0,v}_{q}(X)$, since otherwise Lemma \ref{jz} would imply $fg\in L^1(X)$, a contradiction.
This completes the proof of Theorem \ref{thmfan}.

\section{Proof of Theorem \ref{3log}}\label{sec4}
\setcounter{equation}{0}

To prove Theorem \ref{3log}, it suffices to 
construct a function
\[
f\in B^{0,v}_{q}(X)
\setminus
L(\log^+\!\!L)^{1+v}\Psi(L)(X).
\]

By \eqref{Bbh}, it remains to consider the case \(q=\infty\).
Without loss of generality, we may assume that \(\mu(X)\ge 1\); indeed, if \(\mu(X)<1\), then \(\mu(X)\) is finite and we can normalize to \(\mu(X)=1\). 

Since \((X,d,\mu)\) is nonatomic, we may now select a sequence of pairwise disjoint balls \(\{B_j\}_{j\ge1}\) satisfying
\[
e^{-t_j-1}\leq \mu(B_j)\leq e^{-t_j}, \qquad j\ge1,
\] 
where $\{t_j\}_{j\geq 1}$ is an increasing sequence will be choose later.
Set
$
\lambda_j = \frac{1}{j^2 (t_j)^{1+v}}.
$
Define
\[
a_j(x) = \frac{\chi_{B_j}(x)}{\mu(B_j)}, \quad  f(x) = \sum_{j=1}^{\infty} \lambda_j a_j(x).
\]
Then we have
\[
\supp a_j\subseteq B_j,\quad \|a_j\|_{L^\infty(B_j)} = \frac1{\mu(B_j)}.
\]
Hence, each $a_j$ is an $\infty$-block.
Moreover,
\begin{align*}
&\sum_{j=1}^{\infty}|\lambda_j|\Big(1+ \frac{1}{1+v} \bigl( \log^+(\frac1{\mu(B_j)}) \bigr)^{1+v} \Big)
\lesssim \sum_{j=1}^{\infty} \frac1{ j^2 (t_j)^{1+v} } \Big( 1+\frac1{1+v}(t_j)^{1+v}\Big)<\infty.
\end{align*}
Therefore,
$
f\in B^{0,v}_{\infty}(X).
$

We now prove that
$
f\notin L(\log^+\!\!L)^{1+v}\Psi(L)(X).
$
Indeed, since $\{B_j\}_{j\geq 1}$ are pairwise disjoint, we have
\begin{align*}
&\int_X |f| \bigl( \log(2+|f|) \bigr)^{1+v} \Psi(|f|) \,d\mu\\
&= \sum_{j=1}^{\infty}  \int_{B_j} \frac{|\lambda_j|}{\mu(B_j)} \biggl( \log\Bigl(2+ \frac{|\lambda_j|}{\mu(B_j)} \Bigr) \biggr)^{1+v} \Psi( \frac{|\lambda_j|}{\mu(B_j)} ) \,d\mu
\\
&\geq \sum_{j=1}^{\infty} \frac1{ j^2 (t_j)^{1+v}}
\biggl( \log\Bigl( 2+
\frac{ e^{t_j}}{j^2(t_j)^{1+v}}\Bigr) \biggr)^{1+v}
\Psi( \frac{e^{t_j} }{ j^2(t_j)^{1+v}})\\
&=:I.
\end{align*}

Since $\Psi(t)\to \infty$ as $t\to\infty$, we can choose a sequence $\{t_j\}_{j\geq 1}$ that increases to infinity with $t_j\ge 1$, 
$$
\Psi( \frac{e^{t_j} }{ j^2(t_j)^{1+v}})\geq  j \quad  \text{for all~} j.
$$
On the other hand, as $t_j\to \infty$, we have
$$
\log\Bigl( 2+
\frac{ e^{t_j}}{j^2(t_j)^{1+v}}\Bigr)\approx t_j.
$$
It follows that there exists a constant $j_0>1$ such that 
\begin{align*}
I \gtrsim \sum_{j=j_0}^{\infty} \frac1{ j^2 (t_j)^{1+v}}
t_j^{1+v}
j
=
\sum_{j= j_0}^{\infty}  \frac{1}{j}
= \infty.
\end{align*}
Therefore,
$
f \notin L(\log^+\!\!L)^{1+v}\Psi (L)(X).
$
This completes the proof of Theorem \ref{3log}.

\section{Proof of Theorem \ref{thm}}\label{sec2}
\setcounter{equation}{0}%公式从1开始编号
Let $1<q\le \infty$ and $v>-1$.
To prove Theorem \ref{thm}, it suffices to establish \eqref{key1} and \eqref{key2}. The proof of \eqref{key1} follows the same line as that of \eqref{e1}, with balls replaced by measurable sets. We therefore omit the details.

\subsection{Proof of (\ref{key2})}
By \eqref{key1}, to prove \eqref{key2} it suffices to show 
\begin{equation}\label{B1}
 L(\log^+\!\!L)^{1+v}(X)\subseteq \mathscr{B}^{0,v}_q(X)
\end{equation}
under the assumption that $\mu(X)<\infty$.

Let $v>-1$ and
$
f\in L(\log^+\!\!L)^{1+v}(X).
$
Since 
$$
{\mathscr{B}}^{0,v}_q(X) \subseteq {\mathscr{B}}^{0,v}_p(X),\quad 1<p<q\le \infty,
$$ 
it suffices to prove that
$
f\in \mathscr{B}^{0,v}_{\infty}(X).
$

Set
\begin{align*}
E_0
&=\{x\in X: |f(x)|<2\},\\
E_j &=\{x\in X: 2^j\le |f(x)|<2^{j+1}\},
\quad j\ge1.
\end{align*}

For $j\ge0$, define
\[
\tilde a_j(x) = f(x)\chi_{E_j}(x).
\]
Then we have the decomposition
\[
f = \sum_{j=0}^{\infty} \tilde a_j.
\]

Since $\mu(X)<\infty$, we have $\mu(E_j)<\infty$ for all $j\geq 0$.
Without loss of generality, we may assume $\mu(E_j)>0$ for all $j\geq 0$.
Set $$a_0=\frac{\tilde a_0}{2\mu(E_0)}.$$
Then
$
\tilde a_0= 2\mu(E_0) a_0,
$
and
$$
\supp a_0\subseteq E_0, \quad \|a_0\|_{L^{\infty}(E_0)}\leq \frac{1}{\mu(E_0)}.
$$
Hence, $a_0$ is a generalized $\infty$-block.

Moreover, 
$$
2\mu(E_0)  (1+\frac{1}{1+v}\Big(\log^+ \frac{1}{\mu(E_0)}\Big)^{1+v})<\infty,
$$ 
which implies that $\tilde a_0\in  \mathscr{B}^{0,v}_\infty(X)$. 
Therefore, it remains to prove that
\begin{equation}\label{Om1}
f_1(x) := \sum_{j=1}^{\infty} \tilde a_j(x) \in \mathscr{B}^{0,v}_\infty(X).
\end{equation}

Observe that the sets
$
\{E_j\}_{j\ge1}
$
are pairwise disjoint.
Since $|\tilde a_j(x)|\ge 2^j$ for $x\in E_j$, we have
\[
\log(2+|\tilde a_j(x)|)\ge \log(2+2^j)\ge j\log 2.
\]
Then
\begin{align*}
\int_X |f_1| \Big( \log(2+|f_1|) \Big)^{1+v} \,d\mu
&= \sum_{j=1}^{\infty} \int_{E_j} |\tilde a_j| \Big( \log(2+|\tilde a_j|) \Big)^{1+v} \,d\mu
\\
&\ge \sum_{j=1}^{\infty} \mu(E_j) \,j^{1+v} \,2^j (\log2)^{1+v}.
\end{align*}

Since
$
\|f_1\|_{L(\log^+\!\!L)^{1+v}(X)} \le \|f\|_{L(\log^+\!\!L)^{1+v}(X)} < \infty,
$
we deduce that
\begin{equation}\label{key}
\Lambda := \sum_{j=1}^{\infty} 2^j j^{1+v} \mu(E_j) < \infty.
\end{equation}

Next, for $j\ge1$, set
\[
\lambda_j = 2^{j+1} \mu(E_j), \quad
a_j(x) = \lambda_j^{-1} \tilde a_j(x).
\]
Then we can rewrite $f_1$ as
\[
f_1(x) = \sum_{j=1}^{\infty} \lambda_j a_j(x).
\]
Moreover, for each $j\ge1$, we have
$
\supp a_j \subseteq E_j,
$
and
\begin{align*}
\|a_j\|_{L^\infty(E_j)}
= \lambda_j^{-1} \|\tilde a_j\|_{L^\infty(E_j)}
\le \lambda_j^{-1} 2^{j+1}
= \frac1{\mu(E_j)}.
\end{align*}
Hence, each $a_j$ is a generalized $\infty$-block.
Thus, to show \eqref{Om1}, it remains to verify
\begin{equation}\label{Om2}
\sum_{j=1}^{\infty} |\lambda_j| \Big( 1+\frac1{1+v} \bigl(\log^+(\frac{1}{\mu(E_j)})\bigr)^{1+v}\Big)<\infty.
\end{equation}

Set
\begin{align*}
\Gamma_0 & =\{j\ge1:\mu(E_j)\ge1\},\quad
\Gamma_1  =\{j\ge1:2^{-4j}\le \mu(E_j)\le1\},
\\
\Gamma_2 & =\{j\ge1:\mu(E_j)<2^{-4j}\}.
\end{align*}
Then
\begin{align*}
&\sum_{j=1}^{\infty} |\lambda_j| \Big( 1+ \frac1{1+v} \bigl(\log^+(\frac{1}{\mu(E_j)})\bigr)^{1+v}\Big)\\
&= \sum_{j\in\Gamma_0}|\lambda_j| +\sum_{j\in\Gamma_1} |\lambda_j| \Big(1+\frac1{1+v} \bigl( \log(\frac{1}{\mu(E_j)}) \bigr)^{1+v} \Big)
\\
&\quad
+ \sum_{j\in\Gamma_2} |\lambda_j| \Big( 1+ \frac1{1+v} \bigl( \log(\frac{1}{\mu(E_j)}) \bigr)^{1+v} \Big)
\\
&=:I+II+III.
\end{align*}

For $I$, by \eqref{key}, we have
\begin{align*}
I
&= \sum_{j\in\Gamma_0} 2^{j+1}\mu(E_j) \lesssim \Lambda <\infty.
\end{align*}

For $II$, since
$
\mu(E_j)\ge 2^{-4j}
$
on $\Gamma_1$, we have
\[
\log^+ \frac1{\mu(E_j)} \le 4j\log2.
\]
Hence, by \eqref{key}, we have
\begin{align*}
II \le \sum_{j\in\Gamma_1} 2^{j+1} \mu(E_j) \Big(1+\frac{1}{1+v} \big(4j\log2\big)^{1+v}\Big)\lesssim \Lambda< \infty.
\end{align*}

For $III$, using
\[
\log t \le\frac1\delta t^\delta, \quad \delta>0,\ t>1,
\]
we obtain
\begin{align*}
III
&\le
\sum_{j\in\Gamma_2} 2^{j+1} \mu(E_j) \Big( 1+ \frac{1}{1+v} \big( \frac{1}\delta \mu(E_j)^{-\delta} \big)^{1+v} \Big)
\\
&\le \sum_{j\in\Gamma_2} 2^{j+1} \mu(E_j) + \frac{1}{(1+v)\delta^{1+v}} \sum_{j\in\Gamma_2} 2^{j+1} \mu(E_j)^{1-(1+v)\delta}
\\
&=:III_1+III_2.
\end{align*}

By \eqref{key}, we have
$
III_1 \lesssim \Lambda < \infty.
$
For $III_2$, since
$
\mu(E_j)\le 2^{-4j}
$
on $\Gamma_2$, choosing
$
\delta=\frac1{2(1+v)}
$
yields
\begin{align*}
III_2
&\lesssim \sum_{j\in\Gamma_2} 2^j (2^{-4j})^{1/2}
\lesssim \sum_{j\in\Gamma_2} 2^{-j} < \infty.
\end{align*}
Thus,
\[
III \lesssim III_1+III_2 < \infty.
\]

Therefore, we obtain
\begin{align*}
\sum_{j=1}^{\infty} |\lambda_j| \Big( 1+ \frac1{1+v} \bigl( \log^+(\frac{1}{\mu(E_j)}) \bigr)^{1+v} \Big) 
\lesssim I+II+III < \infty.
\end{align*}
This proves \eqref{Om2} and completes the proof of \eqref{key2}.
Hence, the proof of Theorem \ref{thm} is complete.

\section{Proof of Theorem \ref{B1case}}\label{sec5}
\setcounter{equation}{0}

Let $v>-1$.
We first prove that $B^{0,v}_1(X) \subseteq L^1(X)$ and $\|f\|_{L^1(X)}\le \|f\|_{B^{0,v}_1(X)}$.
Fix
$
f\in B^{0,v}_1(X).
$
Then there exists a decomposition
\[
f=\sum_{j=1}^{\infty}\lambda_j a_j,
\]
where $\lambda_j\in\mathbb C$ and each $a_j$ is a $1$-block satisfying
\[
\supp a_j\subseteq B_j,
\quad
\|a_j\|_{L^1(B_j)}\le1,
\]
and
\[
\sum_{j=1}^{\infty} |\lambda_j| \Big( 1+ \frac1{1+v} \bigl(\log^+(\frac1{\mu(B_j)})\bigr)^{1+v}  \Big) <\infty.
\]

Observe that for any block decomposition, we have
\begin{align*}
\|f\|_{L^1(X)}
\le
\sum_{j=1}^{\infty} |\lambda_j| \|a_j\|_{L^1(B_j)}
\le
\sum_{j=1}^{\infty} |\lambda_j| \Big( 1+ \frac1{1+v} \bigl( \log^+(\frac1{\mu(B_j)}) \bigr)^{1+v}\Big).
\end{align*}
By the definition of the norm
$
\|\cdot\|_{B^{0,v}_1(X)},
$
we get
\[
\|f\|_{L^1(X)}
\le
\|f\|_{B^{0,v}_1(X)}.
\]

We now prove $L^1(X) \subseteq B^{0,v}_1(X)$.
Let
$
f\in L^1(X).
$ 
Without loss of generality, we may assume
$
\|f\|_{L^1(X)}\neq0.
$
We distinguish two cases according to the size of $\mu(X)$.

\textbf{Case 1: $\mu(X)\leq 1$.}

Since $(X,d,\mu)$ is $s$-Ahlfors regular and $\mu(X)\leq 1$,
we have $\operatorname{diam}(X)<\infty$.
Consequently, for any fixed $x_0\in X$,
$
X=B(x_0,2\operatorname{diam}(X)),
$
so that $X$ itself is a ball. 

Set $a=f/\|f\|_{L^1(X)}$, then we have
$$
f=\|f\|_{L^1(X)} a.
$$

Observe that  
$$\supp a \subseteq X,\quad  \|a\|_{L^1(X)}\leq 1.$$
Hence $a$ is a $1$-block.
Moreover,
$$
\|f\|_{L^1(X)} \Big(1+\frac{1}{1+v}\big(\log^+ (\frac{1}{\mu(X)})\big)^{1+v}\Big)<\infty.
$$
Therefore, $f\in B^{0,v}_1(X)$ and 
$$
\|f\|_{B^{0,v}_1(X)} \leq \|f\|_{L^1(X)} \Big(1+\frac{1}{1+v}\big(\log^+ (\frac{1}{\mu(X)})\big)^{1+v}\Big).
$$

\textbf{Case 2: $\mu(X)> 1$.}

Fix \(x_0\in X\). Since $d(x,x_0)<\infty$ for every $x\in X$, we have \(B(x_0,r)\uparrow X\) as \(r\to\infty\).
Thus, there exists \(R>0\) such that
\[
\mu(B(x_0,R)) \ge 1.
\]
Consequently,
\[
X = \bigcup_{j=1}^{\infty} B(x_0, jR).
\] 

Define 
\begin{align*}
  f_1(x)&=f(x)\chi_{B(x_0,R)}(x), \\
  f_j(x)&=f(x)\chi_{B(x_0,jR)\setminus B(x_0,(j-1)R)}(x), \text{~for~} j\geq 2.
\end{align*}
Set
\[
\lambda_j=\|f_j\|_{L^1(X)},\quad 
a_j=
\begin{cases}
\dfrac{f_j}{\|f_j\|_{L^1(X)}}, & \lambda_j\neq0,\\
0, & \lambda_j=0.
\end{cases}
\]
Then 
$$f = \sum_{j=1}^\infty \lambda_j a_j,\quad
\|f\|_{L^1(X)}=\sum_{j=1}^\infty |\lambda_j|.$$

Observe that
$$\supp a_j\subseteq B(x_0,jR),\quad 
\|a_j\|_{L^1(X)}\le1.
$$
Hence, each $a_j$ is a $1$-block.

Furthermore, since
$
B(x_0,R)\subseteq B(x_0,jR),
$
we have
\[
\mu(B(x_0,jR)) \ge \mu(B(x_0,R)) \ge 1.
\]
It follows that
\begin{align*}
\sum_{j=1}^\infty |\lambda_j| \Big(1+ \frac1{1+v} \bigl( \log^+(\frac1{\mu(B(x_0,jR))})\bigr)^{1+v} \Big)
=\sum_{j=1}^\infty |\lambda_j|
 =\|f\|_{L^1(X)}<\infty.
\end{align*}
Hence,
$
f \in B^{0,v}_1(X)
$
and 
$$
\|f\|_{B^{0,v}_1(X)} \leq \|f\|_{L^1(X)}\leq \|f\|_{L^1(X)} \Big(1+\frac{1}{1+v}\big(\log^+ (\frac{1}{\mu(X)})\big)^{1+v}\Big).
$$
This completes the proof of Theorem \ref{B1case}.

\section{Proof of Theorem \ref{Bqcase}}\label{sec6}
\setcounter{equation}{0}

It follows directly from the definition that
$
\tilde{B}_{q,0}(X)\subseteq L^1(X).
$
For $f\in \tilde{B}_{q,0}(X)$, there exists a decomposition
\[
f = \sum_{j=1}^{\infty} \lambda_j a_j,
\]
where each $a_j$ is a $(q,0)$-block supported on a ball $B_j\subseteq X$,
$\lambda_j\in\mathbb{C}$, and
$
\sum_{j=1}^{\infty} |\lambda_j|<\infty.
$

Observe that
$$\|a\|_{L^1(B(x,t))} \leq \mu(B(x,t))^{1/q'} \|a\|_{L^q(B(x,t))} \leq 1.$$
Then we have
$$
\|f\|_{L^1(X)}\leq \sum_{j=1}^{\infty} |\lambda_j|\|a_j\|_{L^1(B_j)}\leq \sum_{j=1}^{\infty} |\lambda_j|<\infty.
$$

By the definition of the norm on $\tilde{B}_{q,0}(X)$, we have
$$
\|f\|_{L^1(X)}\leq \|f\|_{\tilde{B}_{q,0}(X)}.
$$

We now turn to the proof of $\|f\|_{\tilde{B}_{q,0}(X)}\lesssim \|f\|_{L^1(X)}$.
Let
$
f\in L^1(X).
$
By the definition of $\tilde{B}_{q,0}(X)$, we have 
$$
\|f\|_{\tilde{B}_{p,0}(X)} \leq \|f\|_{\tilde{B}_{q,0}(X)},\quad 1<p\leq q\leq \infty.
$$
Thus, it suffices to show that
$
\|f\|_{\tilde{B}_{\infty,0}(X)}\lesssim \|f\|_{L^1(X)}.
$

Set
\[
E_l =\{x\in X:2^l\le |f(x)|<2^{l+1}\},
\quad
l\in\mathbb Z.
\]
Without loss of generality, we may assume $\mu(E_l)>0$ for all $l\in\mathbb Z$.
Then
\begin{equation}\label{fL1}
\|f\|_{L^1(X)} = \sum_{l\in\mathbb Z} \|f\chi_{E_l}\|_{L^1(X)}
\approx
\sum_{l\in\mathbb Z} 2^l\mu(E_l).
\end{equation}

Since $(X,d,\mu)$ is a space of homogeneous type with nonzero Borel regular measure, the Lebesgue differentiation theorem (see \cite{CW71, Tol04}) implies that for a.e. $x\in E_l$, 
$$
\lim_{r\to 0}\frac{\mu(B(x,r)\cap E_l)}{\mu(B(x,r))}=1.
$$
Thus for a.e. 
$x\in E_l$, there exists $r_x>0$ such that
$$
\mu(B(x,r_x)\cap E_l)\geq 1/2 \mu(B(x,r_x)).
$$

Applying the Vitali covering lemma on spaces of homogeneous type (see, for example, \cite[Chapter 3, Theorem 1.2]{CW71}) to the family $
\{B(x,r_{x})\}_{x\in E_l}
$,
we obtain a countable pairwise disjoint subcollection
$
\{B(x_{l,j},r_{l,j})\}_{j\ge1}
$
such that
\[
E_l \subseteq \bigcup_{j\ge1}  B(x_{l,j}, C r_{l,j}),
\]
where $C>0$ is a constant.

Since the balls $\{B(x_{l,j},r_{l,j})\}_{j\in J_l}$ are pairwise disjoint and satisfy
$$
\mu(B(x_{l,j},r_{l,j})) \leq 2 \mu(B(x_{l,j},r_{l,j}) \cap  E_l),
$$
by \eqref{upc}, we obtain
\begin{equation}\label{BE}
\begin{aligned}
\sum_{j\ge1} \mu(B(x_{l,j}, C r_{l,j}))
&\lesssim \sum_{j\ge1} \mu(B(x_{l,j},r_{l,j}))\\
&\lesssim \sum_{j\ge1} \mu(B(x_{l,j},r_{l,j}) \cap  E_l)\\
&\lesssim \mu(E_l).
\end{aligned}
\end{equation}

For $l\in\mathbb Z$, set 
\begin{align*}
F_{l,1}&=B(x_{l,1}, C r_{l,1})\cap E_l,\\
F_{l,j}&=(B(x_{l,j}, C r_{l,j})\setminus F_{l,j-1}) \cap E_l, \quad j \geq 2. 
\end{align*}
Then $\{F_{l,j}\}_{l,j}$ are pairwise disjoint.

Define
\[
\widetilde a_{l,j}(x) := \frac{ f(x)\chi_{F_{l,j}}(x) }{ 2^{l+1}\mu(B(x_{l,j}, C r_{l,j}))}.
\]
Then
\[
f = \sum_{l\in\mathbb Z} \sum_{j\ge1} 2^{l+1} \mu(B(x_{l,j}, C r_{l,j})) \widetilde a_{l,j}.
\]
Observe that
$
\supp \widetilde a_{l,j} \subseteq B(x_{l,j}, C r_{l,j})
$
and
\begin{align*}
\|\widetilde a_{l,j}\|_{L^\infty(B(x_{l,j}, C r_{l,j}))}
=\frac{ \|f\chi_{F_{l,j}}\|_{L^\infty(X)} }{2^{l+1}\mu(B(x_{l,j}, C r_{l,j}))}
\le \frac{1}{\mu(B(x_{l,j}, C r_{l,j}))}.
\end{align*}
Hence each $\widetilde a_{l,j}$ is an $\infty$-block.

Moreover, by \eqref{fL1} and \eqref{BE}, we obtain
\begin{align*}
\sum_{l\in\mathbb Z}\sum_{j\ge1} 2^{l+1}\mu(B(x_{l,j}, C r_{l,j}))
\lesssim \sum_{l\in\mathbb Z} 2^l\mu(E_l)
\lesssim \|f\|_{L^1(X)}.
\end{align*}
Therefore,
$
\|f\|_{ \tilde{B}_{\infty,0}(X)} \lesssim \|f\|_{L^1(X)}.
$
This completes the proof of Theorem \ref{Bqcase}.

\bigskip
\noindent
\textbf{Data availability} No data was used for the research described in the article.

\section*{Declarations}
\noindent
\textbf{Conflict of interest}
 The authors have no conflict of interest.

\end{document}